\documentclass[11pt]{amsart}
\usepackage{amsmath,amscd,amssymb,amsfonts,amsthm}

\usepackage[cmtip,matrix,arrow]{xy}

\newtheorem{theorem}{Theorem}[section]

\newtheorem{proposition}[theorem]{Proposition}

\newtheorem{lemma}[theorem]{Lemma}
\theoremstyle{definition}    
\newtheorem{definition}[theorem]{Definition}
\theoremstyle{remark}

\newtheorem{remark}[theorem]{Remark}

\newtheorem{example}[theorem]{Example}
\renewcommand{\AA}{\mathbb{A}}

\newcommand{\W}{\mathcal{W}}
\newcommand{\V}{\mathcal{V}}
\renewcommand{\L}{\mathcal{L}}

\newcommand{\ca}{\mathcal}

\newcommand{\N}{\mathbb{N}}
\newcommand{\R}{\mathbb{R}}

\newcommand{\Z}{\mathbb{Z}}

\newcommand\pt{\on{pt}}
\renewcommand{\P}{\ca{P}}

\newcommand\lie[1]{\mathfrak{#1}}
\renewcommand{\k}{\lie{k}}
\newcommand{\h}{\lie{h}}
\newcommand{\g}{\lie{g}}

\newcommand{\on}{\operatorname}

\newcommand{\Hom}{ \on{Hom}}

\renewcommand{\ker}{ \on{ker}}

\newcommand\qu{/\kern-.7ex/} 

\newcommand{\lra}{\longrightarrow}
\newcommand{\hra}{\hookrightarrow}

\newcommand{\f}{\frac}

\renewcommand{\l}{\langle}
\renewcommand{\r}{\rangle}

\newcommand{\ti}{\tilde}
\newcommand{\eeq}{\end{eqnarray*}}
\newcommand{\beq}{\begin{eqnarray*}}

\newcommand{\wh}{\widehat}
\newcommand{\wt}{\widetilde}
\newcommand{\mf}{\mathfrak}

\newcommand{\lin}{{\on{lin}}}
\renewcommand{\subset}{\subseteq}
\newcommand{\pol}{{\on{pol}}}
\newcommand{\alg}{{\on{alg}}}

\newcommand{\iz}{\mathsf{i}}

\begin{document}
\sloppy
\title[Quotients of double vector bundles and multigraded bundles]{Quotients of  double vector bundles\\ and multigraded bundles}
\author{Eckhard Meinrenken}
\maketitle

\begin{abstract}
We study quotients of multi-graded bundles, including double vector bundles. Among other things, we show that any such quotient fits into a tower of affine bundles. Applications of the theory include a construction of 
normal bundles for weighted submanifolds, as well as for pairs of submanifolds with clean intersection.  
\end{abstract}

\maketitle
\begin{quote}
	{\it \small To the memory of Kirill Mackenzie.}
\end{quote}
\vskip1cm
\bigskip

\section{Introduction}
Double vector bundles are fiber bundles $D\to M$, described by commutative squares
\begin{equation}\label{eq:dvb}
\xymatrix{ {D} \ar[r] \ar[d] & B \ar[d]\\
	A \ar[r] & M}
\end{equation}
whose sides $A,B$ are vector bundles,  with a certain compatibility between the horizontal and vertical vector bundle structures. 
These objects were introduced by Pradines \cite{pra:rep,pra:fib} in the 1970s in terms of local fiber bundle charts.
Mackenzie \cite{mac:gen} described the compatibilities more intrinsically, as the requirements that the structure 
maps for the horizontal vector bundle should be vector bundle morphisms for the vertical vector bundle structure, and vice versa.  
Later, Grabowski and Rotkiewicz \cite{gra:hig} simplified these conditions to the single requirement that the horizontal and vertical scalar multiplications commute.  

Double structures, and more generally $n$-fold structures for $n>1$, feature prominently in 
Mackenzie's work  \cite{mac:dou,mac:gen}.   Among other things, he discovered a natural action of the symmetric group $S_3$ on the category of double vector bundles through duality operations.  Together with Xu \cite{mac:lie}, he established a canonical isomorphism between $T^*V$ and $T^*V^*$, for any vector bundle $V\to M$. Mackenzie also studied Lie algebroid and Poisson structures on double vector bundles, and unearthed a wealth of interesting  relationships among these structures. In later work, he and his co-authors explored the more difficult theory of triple vector bundles \cite{mac:warp,gra:dua,mac:dua2}, and even $n$-fold vector bundles \cite{gra:dua1}. 

In this note, we will consider \emph{quotients} of double vector bundles $D\to M$ by double subvector bundles $D'\to M$.  
By definition, such a quotient is given by a surjective morphism of double vector bundles $D\to D''$ having $D'$ as its \emph{kernel}
(the elements mapping to $M\subset D''$).  Equivalently, it is given by an exact sequence of double vector bundles over $M$, 
 \begin{equation}\label{eq:d} 0_M\to D'\to D\to D''\to 0_M\end{equation}
where $0_M$ indicates $M$ regarded as a double vector bundle over $M$. In contrast to  ordinary vector bundles, knowledge of the kernel $D'$ does not suffice to construct the quotient $D''$. 
 The additional information needed may be described in terms of a bundle of bigraded Lie algebras $\g\to M$, with summands in bidegrees $(-1,0),(0,-1),(-1,-1)$ given by 
 \begin{equation}\label{eq:gradedlie} \g=\wh{A}\oplus \wh{B}\oplus C.\end{equation}
Here $C\to M$ is the  \emph{core} of the double vector bundle, given as the set of elements mapping to $M$ under both side projections, while  
$\wh{A},\wh{B}\to M$ are the \emph{fat bundles} of Gracia-Saz and Mehta \cite{gra:vba}, given as certain extensions of the side bundles $A,B\to M$ by $B^*\otimes C$ and $A^*\otimes C$, respectively.
The fiberwise bracket on \eqref{eq:gradedlie} is defined by the bilinear pairing, due to Flari and Mackenzie \cite{mac:warp} (see also \cite{me:wei}) 
\[ \l\cdot,\cdot\r_C\colon \wh{A}\times_M \wh{B}\to C. \] 

For degree reasons, $\g$ is a bundle of \emph{nilpotent} Lie algebras; let $G=\exp(\g)$ be the corresponding bundle of simply connected nilpotent Lie groups. 
We will show that there is a canonical action 
\[ G\times_M D\to D\]
of this bundle of Lie groups on the double vector bundle, which in fact is \emph{fiberwise transitive}. We shall prove (cf.~ Theorem \ref{th:quot2}):\medskip

\noindent{\bf Theorem.} The  quotients of $D$ are in bijective correspondence with bigraded Lie subalgebra bundles 
\[ \h=\wh{A}'\oplus \wh{B}'\oplus C'\]
of $\g$, with the property $\wh{A}'\cap (B^*\otimes C)=B^*\otimes C'$, and similarly for $\wh{B}'$. 
 Letting $H\subset G$ be the closed subgroup bundle exponentiating $\h$, we  have that 
$ D'=H\cdot M$ and $D''=D/H$. 
\medskip

We will show that the quotient $D''$ may also be obtained as an \emph{iterated quotient} 
\[ D''=(D/C')/(A'\oplus B')\]
by vector bundle actions: One first takes a quotient by the core action of $C'\subset D'$, followed by an action of the direct sum of side bundles $A'\oplus B'$ of $D'$. However, one still needs the subalgebra $\h$ in order to specify the latter action.

For a concrete context where this construction applies, consider a manifold $M$ together with two submanifolds $N_1,N_2\subset M$, with clean intersection $N=N_1\cap N_2$. In \cite{me:wei}, we defined a 
\emph{double normal bundle} 
\begin{equation}\label{eq:dnb}
\xymatrix{ {\nu(M,N_1,N_2)} \ar[r] \ar[d] & \nu(N_1,N) \ar[d]\\
	\nu(N_2,N)\ar[r] & N}
\end{equation}
whose side bundles $\nu(N_i,N),\ i=1,2$ are the usual normal bundles of $N\subset N_i$. A more detailed discussion may be found  in the thesis \cite{pik:the}. We will see that the double normal bundle is a subquotient of the double tangent bundle $TTM$, in direct generalization of the usual 
 definition  of a normal bundle of a submanifold  as 
$\nu(M,N)=TM|_N/TN$.  Here $D\subset TTM$ is the pre-image of $TN_1|_N\times_N TN_2|_N$ under the side bundle projection to $TM\times_M TM$. 
The subalgebra bundle $\h$ defining the quotient $D''=\nu(M,N_1,N_2)$
may be described explicitly in terms of tangent/vertical 
lifts of vector fields.

\bigskip

There is a similar discussion for the so-called \emph{graded bundles} of Grabowski and Rotkievicz \cite{gra:hig}.  A graded bundle $E\to M$ is a fiber bundle, with typical fiber a negatively graded vector space $\mathsf{E}=\bigoplus_{i=1}^r \mathsf{E}^{-i}$, and with structure group the polynomial diffeomorphisms of $\mathsf{E}$ (not necessarily linear) preserving the grading. Thus, graded bundles may be seen as special cases of non-negatively graded supermanifolds where no odd coordinates are present. Examples of graded bundles include graded \emph{vector} bundles, but also the $r$-th tangent bundle $T_rM=J_0^r(\R,M)$. 
Grabowski and Rotkiewicz \cite{gra:gra} proved that graded bundles may be defined, much more simply, as manifolds $E$ together with a  smooth action of the multiplicative monoid $(\R,\cdot)$. Given a graded subbundle $E'\subset E$ over the same base $M$, one may again ask what additional data are needed in order to define a quotient graded bundle $E''$ with an exact 
sequence of graded bundles 
\[ 0_M\to E'\to E\to E''\to 0_M.\]
Similar to the case of double vector bundles, the answer involves graded Lie subalgebra bundles of a nilpotent graded Lie algebra bundle canonically associated to $E$. Quotient constructions arise, for example, 
in the theory of weightings along submanifolds $N\subset M$; the so-called \emph{weighted normal bundle}
\[ \nu_\W(M,N)\to N\] 
obtained in \cite{loi:wei} is a graded bundle given as a subquotient of the $r$-th tangent bundle $T_rM$.  \medskip

Rather than treating  double vector bundles and graded bundles separately, we will  consider both as special cases of \emph{multi-graded bundles}: That is, manifolds $E$ with $n$ commuting actions of the monoid $(\R,\cdot)$. The two special cases are most relevant in applications, but the general setting is useful as well -- for example, one might 
consider notions of normal bundles for cleanly intersecting submanifolds with weightings \cite{loi:wei}, or $n$-fold clean intersections of submanifolds with $n>2$. We will therefore start our discussion with a review of multi-graded bundles. \medskip


This article is dedicated to the memory of Kirill Mackenzie, who unexpectedly passed away this summer. Over the  years, I benefitted greatly from communications with Kirill, and came to appreciate his deep knowledge and generosity. \medskip

\noindent{\bf Acknowledgements.} I would like to thank Yiannis Loizides and Jeffrey Pike for discussions related to this work, and the referees for a careful reading and valuable comments.

\section{$n$-fold vector bundles and multigraded bundles}
We begin with  background material on multigraded bundles, as needed in this work.  Double vector bundles and graded bundles will be regarded as special cases.

\subsection{$n$-fold vector bundles}
We shall adopt the following approach to  n-fold vector bundles due to Grabowski and Rotkiewicz \cite{gra:hig,gra:gra}. 
Let $(\R,\cdot)$ be the multiplicative monoid of real numbers. Given a vector bundle $E\to M$, the fiberwise scalar multiplication defines a smooth action of this monoid, 
\[ \kappa\colon \R\times E\to E,\ (t,v)\mapsto \kappa_t(v).\]
This action has the regularity property that the map $E\to TE,\ v\mapsto \f{d}{d t}|_{t=0}\kappa_t(v)$ is injective. By a result of 
Nagano \cite{nag:one}, further developed in \cite{gra:hig}, 
any monoid action of $(\R,\cdot)$ on a smooth manifold $E$ with 
this regularity condition determines a unique vector bundle structure on $E$, with $\kappa_t$ as its scalar multiplication. In fact, the vector bundle structure on $E$ is such that the inclusion $E\to TE$ is a morphism of vector bundles. 
This implies, for example, that  subvector bundles of $E$ are exactly the $(\R,\cdot)$-invariant submanifolds, and that smooth maps between vector bundles are vector bundle morphisms if and only if they are $(\R,\cdot)$-equivariant. 

Using the perspective, it was shown in \cite{gra:hig} that double vector bundles are simply manifolds $D$ with two commuting regular $(\R,\cdot)$-actions, a horizontal action $\kappa^h$ and a vertical action $\kappa^v$. 
The side bundles and the base of the double vector bundle are given by 
\[ A=\kappa_0^v(D),\ B=\kappa_0^h(D),\  M=\kappa_0^h\kappa_0^v(D).\] 
More generally, $n$-fold vector bundles are characterized as smooth manifolds $E$
with actions $\kappa\colon \R^n\times E\to E$ of the multiplicative monoid 
\begin{equation}\label{eq:n} (\R^n,\cdot)=(\R,\cdot)\times \cdots \times (\R,\cdot)\end{equation}
such that each component action is regular. 
Typical examples of $n$-fold vector bundles include the $n$-fold tangent bundle $T^nM=T\cdots TM$, or iterations of tangent and cotangent bundles.  As another example, consider $n$ submanifolds $N_1,\ldots,N_n$ of a manifold $M$ with \emph{clean 
intersection}: that is, for all $q\in M$, there are local coordinates  centered at $q$, in such a way that any $N_i$ containing $q$ is given by the vanishing of a subset of these coordinates.
Then  there is an \emph{$n$-fold normal bundle} 
\begin{equation}\label{eq:multinormal} \nu(M,N_1,\ldots,N_n)\to N\end{equation}
with base $N=\bigcap_{i=1}^n N_i$. For $n=2$, this was introduced in \cite{me:wei,pik:the}; see \cite{loi:wei} for general $n$. One of the possible definitions is as a character spectrum: 
Give $C^\infty(M)$ the multi-filtration, where a function has filtration degree $\mathbf{i}=(i_1,\ldots,i_n)$ if it vanishes to order $i_1$ on $N_1$, to order $i_2$ on $N_2$, and so on. Let $\on{gr}C^\infty(M)$ be the associated multi-graded algebra; its summand of degree $\mathbf{i}$ is the space of functions vanishing to order $\mathbf{i}$, modulo sums of functions vanishing to higher order (using the partial ordering on multi-indices). Then 
\eqref{eq:multinormal} may be defined as $\Hom_\alg(\on{gr}C^\infty(M),\R)$. Using evaluation,  
$\on{gr}C^\infty(M)$ is then realized as the algebra of polynomial functions on the $n$-fold normal bundle. We will discuss double normal bundles in Section \ref{sec:dnb} below.

\subsection{Multi-graded bundles}
Grabowski and Rotkiewicz found that it is interesting to drop the regularity assumption: 
\begin{definition}
An \emph{$n$-graded bundle} is a manifold $E$  with a smooth action $\kappa\colon \R^n\times E\to E$ of the multiplicative monoid $(\R^n,\cdot)$, called its \emph{scalar multiplications}. For $n=1$ it is called a \emph{graded bundle}. A \emph{morphism of $n$-graded bundles} $\phi\colon E'\to E$ is a smooth map intertwining the scalar multiplications. 
\end{definition}
By \cite{gra:gra}, an $n$-graded bundle $E$ is a smooth fiber bundle over the fixed point set $M\subset E$ of the action,  with $\kappa_0$ as the bundle projection. 

\begin{example}\label{ex:grvector}
A negatively graded \emph{ vector} bundle $E=E^{-1}\oplus \cdots \oplus E^{-r} \to M$ is a graded bundle, with  $\kappa_t$ given as multiplication by $t^i$ on $E^{-i}$.
We may also regard $E$ as a 2-graded bundle, where the second $(\R,\cdot)$-action is the usual scalar multiplication. 
More generally, $(\Z_{<0})^n$-graded  vector bundles 
\begin{equation}\label{eq:multigradedvector} E=\bigoplus_{i\in \N^n} E^{-i}\end{equation}
are $n$-graded bundles, and may be regarded as $(n+1)$-graded bundles by including the scalar multiplication for the usual vector bundle structure over $M$.  

Conversely, an $n$-graded bundle $E\to M$, together with a vector bundle structure over the same base $M$, is an $n$-graded vector bundle 
if and only if the $(\R^n,\cdot)$-action is linear, i.e., 
commutes with the scalar multiplication from the vector bundle structure. 
\end{example}

\begin{example}\label{ex:rthtangent}
The $r$-th tangent bundle (Ehresmann's bundle of $r$-velocities \cite{ehr:pro})) may be defined as a jet bundle 
\begin{equation}\label{eq:jet}T_rM=J^r_0(\R,M),\end{equation} consisting of $r$-jets of paths $\gamma\colon \R\to M$. It is a graded bundle, with $(\R,\cdot)$-action defined by linear reparametrization of paths.  
Its iterations 
\[ T_{r_1\cdots r_n}M=T_{r_1}\cdots T_{r_n}M\] are examples of $n$-graded bundles. These also admit a more symmetric description: For $ r\in \N$ let $\AA_r$ be the unital graded algebra with a generator $\epsilon$ of degree $-1$, and relation $\epsilon^{r+1}=0$. Then $\AA_{r_1}\otimes \cdots\otimes \AA_{r_n}$ is an $n$-graded algebra, and 
\begin{equation}\label{eq:iterated} T_{r_1\cdots r_n}M=\on{Hom}_{\on{alg}}(C^\infty(M),\AA_{r_1}\otimes \cdots \otimes \AA_{r_n}).\end{equation}
(This is analogous to the characterization of tangent vectors as derivations.) Note that the various `flip isomorphisms' 
(such as $T_{r_1 r_2}M\cong T_{r_2r_1}M$) are obvious from this description. Let us also remark that there is a natural inclusion 
$T_rM\to T_{1,\ldots,1}M=T^rM$, coming from the inclusion $\AA_r\to \AA_1\otimes \cdots \otimes \AA_1$ as the fixed point set of the action of the symmetric group $S_r$ on $\AA_1\otimes \cdots \otimes \AA_1$.  
See \cite[Chapter VIII]{kol:nat}  for more information on the Weil algebra approach to higher tangent bundles.  
\end{example}

\begin{example}\label{ex:lie}
	Let $\g=\bigoplus_{i\in \N^n}\g^{-i}$ be a finite-dimensional $n$-graded Lie algebra, i.e., $[\g^{-i_1},\g^{-i_2}]\subseteq \g^{-i_1-i_2}$. 
	For degree reasons, $\g$ is necessarily nilpotent; let $G=\exp\g$ be the simply connected nilpotent Lie group exponentiating it. 
    Since the  $(\R^n,\cdot)$-action on $\g$ is by Lie algebra morphisms, it exponentiates to an action on $G$ by 
    Lie group morphisms, making $G$ into an $n$-graded bundle over $\pt$. If $\h\subset \g$ is an $n$-graded Lie subalgebra, with corresponding subgroup $H=\exp\h$, then the 
    homogeneous space $G/H$ is an $n$-graded bundle over $\pt$.     Lie groups $G$ with an $(\R,\cdot)$-action by group morphisms  go under various names, for example  \emph{homogeneous Lie groups} (for a detailed discussion, see \cite{fis:qua}).   This example generalizes to Lie algebroids, and in particular to bundles of $n$-graded Lie algebras. 
\end{example}

\subsection{Exact sequences}
Let $E,E'$ be $n$-graded bundles over $M$. For a morphism of $n$-graded bundles $\varphi\colon E'\to E$, with base map the identity, 
we define the range and kernel by 
\[ \on{ran}(\varphi)=\varphi(E'),\ \ \ \ker(\varphi)=\varphi^{-1}(M).\] 
If $\varphi$ has constant rank, 
then these are $(\R^n,\cdot)$-invariant submanifolds, and hence are 
$n$-graded subbundles of $E$ and $E'$, respectively. A sequence of morphisms of graded bundles (over a common base $M$, with base maps the identity) 
\[ \cdots E_\nu\to E_{\nu+1}\to E_{\nu+2}\to \cdots\]
is called \emph{exact} if for each $\nu$, the range of the map $E_\nu\to E_{\nu+1}$ equals the kernel of the map $E_{\nu+1}\to E_{\nu+2}$.  Below, we will encounter a number of examples of exact sequences of $n$-graded bundles.

\subsection{Linear approximation}
Given an $n$-graded bundle $E\to M$, the normal bundle $\nu(E,M)$ relative to the base submanifold is again an $n$-graded  bundle, with scalar multiplications  
obtained by applying the normal bundle functor to $\kappa_t$. Since 
the action $(\R^n,\cdot)$-action on $\nu(E,M)$ 
is linear, the normal bundle is in fact an $n$-fold graded \emph{vector} bundle as in Example \ref{ex:grvector}. We will refer to this bundle as the \emph{linear approximation} of $E$, and use the notation
\begin{equation} \nu(E,M)=E_\lin.\end{equation}
For any multi-index $i=(i_1\ldots,i_n)\in \N^n$, the element $t=(t_1,\ldots,t_n)\in \R^n$ acts as multiplication by 
$t^i=t_1^{i_1}\cdots t_n^{i_n}$ on $E_\lin^{-i}$. 

For an $n$-graded vector bundle \eqref{eq:multigradedvector}, we have $E_\lin\cong E$ canonically. 
For the $r$-th tangent bundle 
\eqref{eq:jet} we have that 
\[ (T_rM)_\lin^{-i}=\begin{cases}
TM & i\le r\\0 & \mbox{otherwise.}
\end{cases}\] 
The same description holds for the 
iterations \eqref{eq:iterated}, interpreting $r$ as a multi-index 
 $(r_1,\ldots,r_n)$, and using the partial ordering on multi-indices $k\in \Z^n$, where $k\ge k'$ if and only if $k_a\ge k_a'$ for all $a$.  For every morphism of graded bundles $\varphi\colon E'\to E$, applying the normal bundle functor gives a morphism of graded vector bundles $\varphi_\lin\colon E'_\lin \to E_\lin$ called the linear approximation of $\varphi$. 
 
\begin{theorem}\cite{gra:hig,gra:gra} For every graded bundle $E\to M$, there exists an isomorphism of graded bundles 
\[ \varphi\colon E_\lin\stackrel{\cong}{\lra} E,\] 
with linear  approximation 	the identity. 
\end{theorem}
The isomorphism is not unique (only its restriction to $E_\lin^{-i}$, with $i$ maximal subject to the condition that  $E_\lin^{-i}\neq 0$, is unique). We refer to the choice of $\varphi$ as a \emph{linearization} of $E$.  
\begin{example}
For a double vector bundle $D$, as in \eqref{eq:dvb}, we have that
\[  D_\lin^{-1,0}=A,\ \ \ D_\lin^{0,-1}=B,\ \ \
 D_\lin^{-1,-1}=C\]
 (where $C$ is the core).  
Hence, a linearization is a choice of isomorphism of double vector bundles 
\[ \varphi\colon A\oplus B\oplus C\to D\]
whose linear approximation is the identity. The restriction of $\phi$ to $C\subset D$ is uniquely determined by this property. A linearization of a double vector bundle is also known as a \emph{splitting} or \emph{decomposition}. For direct proofs, see e.g.
\cite{del:geo,jot:mult,me:wei}. 
\end{example}

\subsection{Polynomial functions}
Let $E\to M$ be an $n$-graded bundle, with scalar multiplication $\kappa_t,\ t\in \R^n$. 
A function
$f\in C^\infty(E)$ is called a \emph{homogeneous polynomial function of degree $k\in \Z^n$} if 
\[ (\kappa_{t})^*f=t^k f\]
for all $t\in (\R-\{0\})^n$. Let  $C^\infty_\pol(E)^k$ be the space of homogeneous polynomials of degree $k$, and 
\[ C^\infty_\pol(E)=\bigoplus_{k\in \Z^n} C^\infty_\pol(E)^k\]
of polynomial functions on $E$. Using a linearization of $E$, we see that the  space $C^\infty_\pol(E)^k$ is non-trivial only if $k\ge 0$, and is the space of sections of a 
vector bundle  $\ca{P}(E)^k$. Taking a direct sum over $k$, we obtain a graded algebra bundle 
\begin{equation}\label{eq:pols} \ca{P}(E)=\bigoplus_{k\in \Z^n}\ca{P}(E)^k.\end{equation}
\begin{example}
Let $E=\bigoplus_{i\in \N^n}E^{-i}$ be a $(\Z_{<0})^n$-graded \emph{vector} bundle. Then 
\[ \ca{P}(E)^k=\bigoplus_{\{k_i\}} \bigotimes_{i}\on{Sym}^{k_i}(E^{-i})^*\]
where the sum is over sequences $\{k_i,\ i\in \N^n\}$ such that $\sum k_i\cdot i=k$, and $\on{Sym}$ indicates symmetric powers. This describes the situation for a general $n$-graded bundle $E$ by choice  of a linearization. 
\end{example}

\begin{example}
	For a double vector bundle \eqref{eq:dvb}, the algebra bundle $\P(D)$ is generated (as a \emph{unital} algebra bundle) by its components $\P(D)^{1,0},\ \P(D)^{0,1},\ \P(D)^{1,1}$. 
	There are canonical isomorphisms $\P(D)^{1,0}=A^*$ and $\P(D)^{0,1}=B^*$, and an exact sequence
	\[ 0_M\to A^*\otimes B^*\to \P(D)^{1,1}\to C^*\to 0_M,\] 	
	where the first map is given by multiplication in $\P(D)$. The bundle of double-linear functions $\P(D)^{1,1}$ is the `fat bundle' associated to $C^*$. This bundle plays an  important role in the theory of DVB-sequences of \cite{che:dou}, where it is shown that a 
	double vector bundle with side bundles $A,B$ and core $C$ is equivalent to such an exact sequence. Indeed, the exact sequence allows us to reconstruct the algebra bundle $\P(D)$, which in turn determines  $D=\Hom_\alg(\P(D),\R)$ (the fiberwise character spectrum). 	See also \cite{me:wei} for further details and related constructions. 
\end{example}

Morphisms of graded bundles $E\to E'$ over $M$ give morphisms of graded 
algebra bundles $\ca{P}(E')\to \ca{P}(E)$ by pullback of functions. Letting $0_M=M\times \{0\}$ (regarded as an $n$-graded bundle), we have  $\ca{P}(0_M)=M\times \R$. 
\begin{definition}
	The \emph{augmentation ideal} 
	$\P^+(E)$ is the kernel of the map $\P(E)\to \P(0_M)$ given by the inclusion $0_M\to E$.
\end{definition}
The sections of the augmentation ideal are the polynomial functions vanishing along $M\subset E$.

\subsection{Vertical vector fields}
The grading on polynomial functions induces gradings on all polynomial tensor fields on $E$. For the quotient construction, we are mainly interested in  vertical vector fields. 
Let 
\[ \V(E)\to M\] 
be the $\Z^n$-graded Lie algebra bundle of fiberwise graded derivations of $\P(E)$. The sections of 
$\V(E)^k$ are the  
\emph{vertical} vector fields $X$ on $E$ satisfying $\kappa_t^* X=t^k X$ for $t\in (\R-\{0\})^n$. 
The Lie algebra bundle 
$\V(E)$ is a graded  $\P(E)$-module. 

\begin{example}
	Let $E=\bigoplus_{i\in \N^n}E^{-i}$ be a $(\Z_{<0})^n$-graded \emph{vector} bundle. Then there is a degree preserving inclusion 
\[ E\hra \V(E)\]
as \emph{(fiberwise) constant vertical vector fields}. Since these span the vertical tangent space everywhere, we have 
$\V(E)=\P(E)\otimes E$ in this case. For a general $n$-graded bundle, such a decomposition depends on the choice 	of a linearization. 
\end{example}

\begin{example}
	For a double vector bundle \eqref{eq:dvb},  the space $\V(D)^k$ is trivial unless 
	$k\ge (-1,-1)$. A vector field of homogeneity $(-1,-1)$ restricts to the core, and the restriction map gives an isomorphism 
	\[\V(D)^{-1,-1}\cong C.\] 
	Similarly, a vector field of  homogeneity $(-1,0)$ restricts to the side bundle $A$; the kernel of the restriction map 
	is $\P(D)^{0,1}\otimes \V(D)^{-1,-1}=B^*\otimes C$ (linear functions on $B$ times 
	 core vector fields).  We may take
	\[ \V(D)^{-1,0}=\wh{A},\ \ \ \V(D)^{0,-1}=\wh{B},\] 
	as a definition of the \emph{fat bundles} described in the introduction.
\end{example}

Recall the augmentation ideal $\P^+(E)$, and let 
\[ \V^+(E)=\P^+(E)\V(E).\] 

\begin{lemma}
	Restriction to $M\subset E$ gives an exact sequence of graded bundles 
	\begin{equation}\label{eq:exactve} 
	0_M\to \V^+(E)\to \V(E)\to E_\lin\to 0_M.\end{equation}
	In particular, the sections of $\V^+(E)$ are the vertical polynomial vector fields vanishing along $M\subset E$. 
\end{lemma}
\begin{proof}
	Choosing a linearization $E\cong E_\lin$, we see that the restriction map $\V(E)\to E_\lin\subset TE|_M$ is surjective, with kernel the vertical vector fields vanishing along $M$.  
\end{proof}

\subsection{The bundle of nilpotent Lie groups}
The $n$-graded Lie algebra bundle $\V(E)$ has a subbundle, spanned by the components of strictly negative degree
\[ \g=\V(E)^{<0}=\bigoplus_{i\in \N^n} \V(E)^{-i}\]
Let $\g^+=\g\cap \V^+(E)$. The sections of $\g$ are polynomial vector fields on $E$ of strictly negative degree, and those in 
$\g^+$ are furthermore zero along $M\subset E$. 

\begin{lemma}
	The bundle $\g$ has finite rank and generates $\V(E)$ as a $\P(E)$-module. Restriction to $M\subset E$ gives an exact sequence of graded bundles 
	\begin{equation}\label{eq:exactve} 
	0_M\to \g^+\to \g\to E_\lin\to 0_M.\end{equation}
The action of $\g$ on $E$ is fiberwise transitive. (That is, the corresponding vector fields span the vertical subbundle of $TE$ everywhere.) 
\end{lemma}
\begin{proof}
Choosing a linearization $E\cong E_\lin$, and regarding sections of $E_\lin$ as fiberwise constant vector fields on $E_\lin$, we obtain an isomorphism 
\[ \V(E)=\P(E)\otimes E_\lin.\] 	
The statement is clear from this linearized model; in particular 
\begin{equation}\label{eq:g} \g=\bigoplus_{j-i<0} \P(E)^j\otimes E_\lin^{-i}\end{equation}
has finite rank, and it generates $\V(E)$ since already $E_\lin$ does. Similarly, the vector fields corresponding to $\g$ span the vertical space of $E\to M$, since already $E_\lin\subset \g$ does. 
\end{proof}	

\begin{example}
For a double vector bundle \eqref{eq:dvb}, this is the Lie algebra bundle 
\[ \g=\wh{A}\oplus \wh{B}\oplus C\] 
described in the introduction. The bracket restricts to a (possibly degenerate) pairing  
$\wh{A}\otimes \wh{B}=\V^{-1,0}(D)\otimes \V^{0,-1}\to \V^{-1,-1}(D)=C$, which is the warp pairing of Flari-Mackenzie \cite{mac:warp}. Note that this pairing determines the bracket, for degree reasons. Further details will be given in Section \ref{subsec:nilp} below. 
\end{example}

For degree reasons, the Lie algebra bundle $\g$ is nilpotent. By fiberwise exponentiation, we obtain an $n$-graded  Lie group bundle  (see Example \ref{ex:lie}) 
\[ G=\exp(\g)\to M,\] 
with an $n$-graded subbundle $G^+=\exp(\g^+)$.  
\begin{lemma}
The fiberwise $\g$-action on the $n$-graded bundle $E\to M$ exponentiates to a fiberwise action of $G$. 
\end{lemma}
\begin{proof}
Recall that a Lie algebra action exponentiates to an action of the corresponding connected, simply connected Lie group if and only if its generating vector fields are complete. We will use this fact fiberwise, so we may assume $M=\pt$, 
with $E$ a graded vector space $\mathsf{E}=\bigoplus_{i\in \N^n} \mathsf{E}^{-i}$.  After passage to the diagonal action of $(\R,\cdot)\subset (\R^n,\cdot)$ (which does not change the space of polynomial vector fields of negative degree) it suffices to consider the case 
$n=1$.  Letting  $X=\sum_{i\in \N}X_i\in\mf{X}(\mathsf{E})$, with $X_i$ homogeneous of degree $-i<0$, 
we want to show that $X$ is complete. 
Choose a smooth function 
$\varrho\colon \mathsf{E}\to \R_{>0}$, homogeneous of degree $1$ outside a compact subset. 
Then
$\L_{X_i}\log\varrho$ is homogeneous of degree $-i< 0$ outside a compact subset, and hence is bounded on 
$\mathsf{E}$. Consequently, $\L_X\log(\varrho)$ 
is bounded on $\mathsf{E}$. 
This implies that the flow of $X$ cannot go to infinity in finite time, so $X$ is complete. 
\end{proof}

Since the $\g$-action on $E$ is fiberwise transitive, with stabilizer of $M$ the subbundle $\g^+$, it follows 
that the $G$-action on $E$ is fiberwise transitive, with stabilizer of $M$ the subbundle $G^+$. 
We see in particular that  every $n$-graded bundle is a homogeneous space for an $n$-graded Lie group bundle:
\[ E=G/G^+.\]
See Section \ref{subsec:nilp} for a description of the group bundle $G$ associated to double vector bundles. 

\begin{remark}
If $E\to M$ is an ordinary vector bundle (regarded as a $1$-graded bundle), we have that $\g= E$, 
where elements of $E$ are identified with constant vector fields on fibers, and $\g^+=0$ (the zero bundle). 
The corresponding group bundle is $G=E$ with the group multiplication given by addition, and the $G$-action on $E$ is just the translation 
action. 
\end{remark}

Let $i\in \N^n$ be maximal subject to the condition   $E_\lin^{-i}\neq 0$. Then $E_\lin^{-i}\cong \g^{-i}$ is in the center of $\g$, 
and its image under $\exp$ defines a central subgroup bundle 
\[ E_\lin^{-i}\subset \on{Cent}(G)\]
(with group structure the vector bundle addition). In terms of a linearization of $E_\lin\cong E$,  its action 
is given simply the translation. The  quotient by this action is again an $n$-graded bundle, with linear approximation 
\[ (E/E_\lin^{-i})_\lin=\bigoplus_{j\neq i} E_\lin^{-j}.\]

\subsection{Splitting short exact sequences}
A short exact sequence 
\begin{equation}\label{eq:exactsequence}
0_M\lra E'\stackrel{i}{\lra} E\stackrel{\pi}{\lra} E''\lra 0_M
\end{equation}
(where $0_M=M\times \{0\}$ denotes the zero bundle over $M$, regarded as an $n$-graded bundle) is equivalent to a surjective morphism of graded bundles $\pi\colon E\to E''$, with $i\colon E'\to E$ the inclusion of its kernel.  A \emph{splitting} of a short  exact sequence is a 
morphism  $j\colon E''\to E$ such that $\pi\circ j=\on{id}_{E''}$. As in the case of vector bundles, short exact sequences always split: 
\begin{proposition}\label{prop:existenceofsplittings} Every short exact sequence \eqref{eq:exactsequence} of $n$-graded bundles admits a splitting.  
\end{proposition}

\begin{proof}
	Let $i\in \N^n$ be maximal subject to the condition $E_\lin^{-i}\neq 0$. The exact sequence for 
	$E$ restricts to an exact sequence of vector bundles 
	\begin{equation}\label{eq:sequencevector}  0_M\to (E_\lin')^{-i}\to E_\lin^{-i}\stackrel{\pi}{\lra} (E_\lin'')^{-i}  \to 0_M.\end{equation}	
	Taking the quotient by the respective vector bundle actions, this produces an exact sequence of $n$-graded bundles 
	\begin{equation}\label{eq:sequencequotient}
	0_M\to E'/(E_\lin')^{-i}\to E/E_\lin^{-i}\stackrel{\pi}{\lra} E''/(E_\lin'')^{-i}  \to 0_M.\end{equation}
	By induction on  the rank of $E$,  \eqref{eq:sequencequotient} admits a splitting; fix such a splitting and also a splitting of \eqref{eq:sequencevector}. 
	To prove the proposition, it suffices to find isomorphisms 
	\begin{equation}\label{eq:ismo}
	E\cong (E/E_\lin^{-i})\times_M E_\lin^{-i},\ \ \ \ 
	E''\cong (E''/(E''_\lin)^{-i})\times_M (E''_\lin)^{-i}\end{equation}
	that are compatible with the projection $\pi$, in the sense that
	$\pi(x,w)=(\pi(x),\pi(w))$. Let  
	\[ q\colon E\to E/E_\lin^{-i},\ \ \ q''\colon E''\to E''/(E''_\lin)^{-i}\] 
	be the two quotient maps. The choice of isomorphisms \eqref{eq:ismo} amount to splittings
	\[ s\colon E/E_\lin^{-i}\to E,\ \ s''\colon E''/(E''_\lin)^{-i}\to E'',\] i.e., 
	$q\circ s=\on{id},\ q''\circ s''=\on{id}$. Compatibility with $\pi$   is equivalent to 
	commutativity of the diagram 
	\[\xymatrixcolsep{5pc}
\xymatrix{ {E/E_\lin^{-i}} \ar[r]_s \ar[d]_\pi & E \ar[d]^\pi\\
	E''/(E_\lin'')^{-i} \ar[r]_s & E''}
\]
	Start with arbitrary splittings $s,s''$ (e.g., defined  by linearizations of $E,E''$), and  
	consider the difference
	\[ f=\pi\circ s-s''\circ \pi\colon E/E_\lin^{-i}\to E''.\]
	%
From 	$q''\circ f=0$, we see that $f$ takes values in $(E''_\lin)^{-i}\subset E''$. The given splitting of 
	\eqref{eq:sequencevector} lifts  $f$ to a 
	morphism 
	\[ \widehat{f}\colon E/E_\lin^{-i}\to E_\lin^{-i}.\] Then 
	$\wt{s}=(-\widehat{f})\cdot s$ is again a splitting of $q$ (the dot indicates the action of the vector bundle 
	$E_\lin^{-i}$ on $E$), and $\pi\circ \wt{s}-s''\circ \pi=0$, as desired. 
\end{proof}

\section{Quotient constructions}
In contrast to vector bundles, an $n$-graded subbundle of an $n$-graded bundle does not 
determine a quotient  in a canonical way. In this section, we will clarify what additional structure is needed. Throughout, $E$ will denote an $n$-graded bundle over $M$, with base projection $p\colon E\to M$.

\subsection{Criterion for quotients}
Consider a quotient of $E$, described by a short exact sequence
\[ 0_M\to E'\to E\to E''\to 0_M.\]
The tangent bundle to the quotient map $\pi\colon E\to E''$ is a $\kappa_t$-invariant integrable distribution $T_\pi E\subset TE$.   View the elements of $\V(E)$ as vector fields on the respective fiber of $E\to M$, and let 
\[ \V(E)_\pi\subset \V(E)\] 
be the subset for which the vector field is tangent to the fibers of $\pi$. Thus, the action map 
\begin{equation}\label{eq:actionmap} \varrho\colon p^*\V(E)\to TE\end{equation}
takes $p^*\V(E)_\pi$  to $T_\pi E$. Put $\V^+(E)_\pi=\V^+(E)\cap \V(E)_\pi$. 
\begin{lemma}\label{lem:prep}
The subset $\V(E)_\pi\subset \V(E)$ is a bundle of graded Lie subalgebras, and of $\P(E)$-submodules. It has the property  
\[ \V^+(E)_\pi=\P^+(E)\V(E)_\pi.\]
\end{lemma}
\begin{proof}
	By Proposition \ref{prop:existenceofsplittings}, we may choose a splitting $E\cong E'\times_M E''$ (with $\pi$ given by projection to the second factor). This   
	identifies $\P(E)=\P(E')\otimes \P(E'')$, and determines a direct sum decomposition decomposition as $\P(E)$-submodules
	\[\V(E)=\big(\P(E')\otimes \V(E'')\big)\oplus \big(\P(E'')\otimes \V(E')\big).\]
	The second summand in this decomposition is $\V(E)_\pi$;
   in particular, $\V(E)_\pi$ is a graded subbundle. It is immediate that 
	$\V(E)_\pi$ is closed under the bracket operation of $\V(E)$. We see furthermore 
	\[\V^+(E)_\pi =(\P(E'')\otimes \V^+(E'))+ (\P^+(E'')\otimes \V(E'));\]
	since $\P^+(E)=\P^+(E')\otimes \P(E'')+\P(E')\otimes \P^+(E'')$, but this is the same as $\P^+(E)\cdot \V(E)_\pi$. 
\end{proof}
The following result shows that the property observed in Lemma \ref{lem:prep} actually 
characterizes the subbundles of the form $\V(E)_\pi$.

\begin{theorem}\label{th:quotients}
	Suppose $\L\subset \V(E)$ is a graded Lie subalgebra bundle, which is also a $\P(E)$ submodule. Then $\L=\V(E)_\pi$ for a 
	(unique) quotient map $\pi\colon E\to E''$ if and only if the submodule $\L^+=\L\cap \V^+(E)$ satisfies 
	\begin{equation}\label{eq:condition} \L^+=\P^+(E)\, \L.\end{equation}
	In terms of the Lie algebra bundle $\h=\L\cap \V(E)^{<0}$, and the corresponding Lie group bundle $H$, we have that 
	\[ E'=H\cdot M\subset E,\ \ \ \ E''=E/H.\]
\end{theorem}
\begin{proof}
Lemma \ref{lem:prep} shows that this condition is satisfied if $\L$ is of the form $\L=\V(E)_\pi$. For the converse,  suppose that $\L\subset \V(E)$ is a graded $\P(E)$-submodule, closed under brackets, and satisfying \eqref{eq:condition}. 
Let $S\subset \L$ be a graded subbundle complementary to $\L^+$, so that 
\[ \L=S\oplus \L^+.\]
Note that $\L$ and $\L^+$ coincide in non-negative degrees (since $\V(E)$ and $\V^+(E)$ agree in non-negative degrees, by the exact sequence \eqref{eq:exactve}). Hence $S$ is non-trivial only in strictly negative degree, and in particular is of finite rank. 
Using \eqref{eq:condition}, any  $x\in (\L^+)^k$ may be written as
\begin{equation}\label{eq:xexpression} x=\sum_a p_a x_a\end{equation}
where 
$x_a \in \L^{k_a}$ with $k_a<k$, so $p_a\in \P^+(E)$. 
Each $x_a$ appearing in \eqref{eq:xexpression} may be decomposed into its $S$ and $\L^+$ components; 
to the $\L^+$ component we may apply  \eqref{eq:condition} again, and so on. Lowering the degrees of $x_a$ in this way, we eventually obtain an expression \eqref{eq:xexpression} where all $x_a\in S$. That is, 
\[\L^+=\P^+(E)\,S.\]
Hence also $\L=\P(E)\,S$, which implies that the image of $p^*\L$ under the action map \eqref{eq:actionmap} coincides with the image of $p^*S$. Along 
$M\subset E$, the map $\varrho$ is the surjection $\V(E)\to E_\lin\subset TE$, with kernel $\V^+(E)$. Hence it restricts to an injective map $S\to E_\lin$ along $M$. We conclude that $\varrho\big|_{p^*S}\colon p^*S\to TE$ is injective along $M$, hence also on some 
open neighborhood of $M$ in $E$. By homogeneity, it is injective everywhere. 

We conclude that $\Gamma(\L)$ (regarded as a subalgebra of $\mf{X}(E)$ by the action $\varrho$) 
defines a constant rank, $\kappa_t$-invariant distribution on $E$. Since 
$\Gamma(\L)$ is involutive, this distribution is Frobenius integrable, and so defines a foliation of $E$. Let 
\[ \h=\L\cap \V(E)^{<0}=\L\cap \g\subset \g\]
be the nilpotent Lie subalgebra bundle given as the summand of strictly negative degree, and $H\subset G$ the corresponding Lie group bundle. Since $S\subset \h$, the orbits of the Lie algebra action of $\L$ coincide with those of the $H$-action, 
\[  H\times_M E\to E.\]
Since this action preserves $n$-gradings, the quotient space $E''=E/H$ is an $n$-graded manifold, with $\pi$ as its quotient map.  The kernel of the quotient map is the leaf through $M$, i.e. $E'=H\cdot M$. 
\end{proof}

\begin{remark}
	The condition \eqref{eq:condition} can be restated in terms of the subalgebra bundle $\h=\L\cap \g$ (the strictly negative summands of $\L$), as the requirement that 
	\begin{equation}\label{eq:condition1}
	\h\cap \P^+(E)\g=\P^+(E)\h\cap \g.\end{equation}
	The exact sequence \eqref{eq:exactve} restricts to an exact sequence 
	of graded vector bundles
	\begin{equation}\label{eq:exacth} 0_M\to \h^+\to  \h\to E_\lin'\to 0_M.\end{equation}
\end{remark}

%

\begin{remark}\label{lem:smallerh}
Suppose $\k\subset \g$ is a graded Lie subalgebra bundle satisfying the weaker condition 
\begin{equation}\label{eq:weakercondition}
\k\cap \P^+(E)\g\subset \P^+(E)\k\cap \g.
\end{equation}	
Then  $\h=\P(E)\k\cap \g$ satisfies \eqref{eq:condition1}, and so determines a 
quotient $\pi\colon E\to E''=E/K=E/H$, with kernel $E'$. 
\end{remark}

\subsection{Iterated quotients}
Rather than taking a quotient of $E$ by the action of a nilpotent group bundle, we may also take an iterated quotient by vector bundle actions. In other words, every quotient map factors as a tower of affine bundles. 
		
\begin{theorem}\label{th:stages}
Every quotient map $\pi\colon E\to E''$ of $n$-graded bundles over $M$ factors as a sequence of morphisms 
\[ E\to E_1\to E_2\to \cdots \to E''\]
where each map is a quotient by a free and proper vector bundle action. 
\end{theorem} 
\begin{proof}
	Let $H=\exp\h\subset G$ be as in  Theorem \ref{th:quotients}, so that $E''=E/H$.
Let $i\in \N^n$ be maximal subject to the condition that $\h^{-i}\neq 0$. For degree reasons, $\h^{-i}$ is contained in the center of $\h$. Its image under $\exp\colon \h\to H$ is a central subgroup bundle $\exp\h^{-i}$, with group structure the vector bundle addition 
of $\h^{-i}$. Let
\[ E_1=E/\exp(\h^{-i})\] 
be the 
quotient under the vector bundle action of $\exp\h^{-i}$. 

The quotient by $\h^{-i}\subset \h$ is the same as the quotient by $\P(E)\h^{-i}\cap \h$, which is a normal Lie subalgebra bundle of $\h$. Hence, the quotient $\h_1=\h/(\P(E)\h^{-i}\cap \h)$ acts on $E_1$, and $E_1/H_1=E''$. 
By induction, the quotient map $E_1\to E''$ factors as a sequence of quotients by  free and proper vector bundle actions. 
\end{proof}

\begin{example}
Let $E$ be a graded bundle (with $n=1$) of order $r$, i.e. $E^{-i}=0$ for $i>r$, 
and take $\h=\g=\V(E)^{<0}$. The quotient map to $E/G=M$ factors as a sequence of quotient maps 
\[ E\to E_1\to E_2\to \ldots \to M\]
where the first quotient is by a vector bundle action of $E_\lin^{-r}$, the second quotient is by a vector bundle action of $E_\lin^{-r+1}$, and so on. In particular, the tower of higher tangent bundles 
\[ T_rM\to T_{r-1}M\to \cdots \to TM\to M\]
is of this form; here each arrow is a quotient map by an action of $TM$. 
\end{example}

\begin{example}
	The bundle projection $D\to M$ of a double vector bundle \eqref{eq:dvb} may be factored as 
	\[ D\to A\times_M B\to A\to M\]
	where the first map is a quotient by the core $C$, the second map is  a quotient by $B$, and the last map is a quotient by $A$. Note that there are no natural actiuons of $A,B$ on $D$. 
\end{example}

\section{Double vector bundles}

In this section, we specialize the general quotient construction to double vector bundles (repeating a few facts already mentioned). 
For background on double vector bundles, see e.g. \cite{gra:hig,gra:vb,mac:gen,me:wei}.

\subsection{Linear approximation}
For any double vector bundle 
\begin{equation}\label{eq:dvb1}
	\xymatrix{ {D} \ar[r] \ar[d] & B \ar[d]\\
		A \ar[r] & M}
	\end{equation}
the map $D\to A\times_M B$ given by the 
horizontal and vertical base projections is a surjective submersion, with kernel the {core}  $C\to M$ of the double vector bundle. 
Thus, by definition there is a short exact sequence of double vector bundles \cite[Section 2.2]{gra:vb}
\[ 0_M\to C\to D\to A\times_M B\to 0_M.\]
Since the horizontal and vertical scalar multiplications coincide on  $C$, the core may be regarded as a vector bundle, $C\to M$.  The linear approximation of $D$ is a bigraded vector bundle $D_\lin\to M$ with summands
	 \[D_\lin^{-1,0}=A,\ \ \ D_\lin^{0,-1}=B,\ \ \ D_\lin^{-1,-1}=C.\]
The scalar multiplications of $t=(t_1,t_2)\in \R^2$ on $D_\lin$ are given by 
\[ \kappa_t (a,b,c)=(t_1 a,\ t_2 b,\ t_1t_2 c),\ \ \ (a,b,c)\in A\times_M B\times_M C;\]
after a choice of linearization $D_\lin\cong D$ this serves as a model for $D$. 
 
\subsection{The nilpotent group bundle $G\to M$}	 \label{subsec:nilp}
The Lie algebra bundle $\g=\V(D)^{<0}$ has components in the same degrees as $D_\lin$. The bundle map 
$\g\to D_\lin$ is an isomorphism in bidegree $(-1,-1)$ (since $i=(1,1)\in \N^2$ is the unique maximal index for which $\g^{-i}\neq 0$); hence we have 
\[\g^{-1,-1}=C.\] 
Sections of $C$ define vector fields on $D$ of homogeneity $-1,-1$, these are the \emph{core vector fields}. The action $C\times_M D\to D$ of the corresponding subgroup bundle $C\cong \exp(\g^{-1,-1})\subset G$ on $D$ is known as the \emph{core action} \cite{mac:warp}. It is a vector bundle action, given  in terms of a linearization by 
\[ c_1\cdot (a,b,c)=(a,b,c+c_1).\] 
The other non-zero components of $\g$ are 
\[ \g^{-1,0}=\wh{A},\ \ \ \g^{0,-1}=\wh{B};\] 
these are the so-called \emph{fat bundles} of Gracia-Saz and Mehta \cite{gra:vba}.
The sections of $\wh{A}$ are realized as vector fields on $D$ that are 
invariant in the vertical direction and constant  in the horizontal direction; 
the map $\g^{-1,0}\to D_\lin^{-1,0}$ becomes a bundle map $\wh{A}\to A$, which is realized on the level of sections by restrictions of vector fields to $A\subset M$. The kernel of this bundle map is 
\[  (\g^+)^{-1,0}=\P(D)^{0,1}\otimes \g^{-1,-1}=B^*\otimes C,\]
realized on the level of sections by products of core vector fields with linear functions on $B$ (pulled back by the horizontal projection $D\to B$). We obtain the short exact sequence (see \cite[Section 4]{gra:vb})
\begin{equation}\label{eq:ahat} 0_M\to B^*\otimes C\stackrel{\iz_{\wh{A}}}{\lra}\wh{A}\lra A\to 0_M.\end{equation}
A similar discussion applies to $\wh{B}$, leading to an exact sequence 
\begin{equation}\label{eq:bhat} 0_M\to A^*\otimes C  \stackrel{\iz_{\wh{B}}}{\lra}      \wh{B}\lra B\to 0_M.\end{equation}
For degree reasons, both $\g^{-1,0},\g^{0,-1}$ are abelian Lie subalgebra bundles of $\g$. The Lie bracket between these two summands
defines a bilinear pairing 
\[ \l\cdot,\cdot\r\colon \wh{A}\times_M\wh{B}\to C,\]
introduced in the work of Flari-Mackenzie \cite{mac:warp} under the name of \emph{warps}  (our notation follows \cite{me:wei}). 
Hence,  
\[ \g= \wh{A}\oplus \wh{B}\oplus C\]
is a bundle of Heisenberg Lie algebras for this bilinear form. The group bundle $G\to M$ is equal to $\g\to M$ as a fiber bundle, 
with the group structure given by 
\[ (\wh{a}_1,\wh{b}_1,c_1)\cdot (\wh{a}_2,\wh{b}_2,c_2)= \big(\wh{a}_1+\wh{a}_2,\wh{b}_1+\wh{b}_1,c_1+c_2+
\l \wh{a}_1,\wh{b}_2\r_C-\l\wh{a}_2,\wh{b}_1\r_C\big)
\]
for $\wh{a}_i\in \wh{A},\ \wh{b}_i\in \wh{B},\ c_i\in C$ (all with the same base point). 
For an explicit description of the action $ G\times_M D\to D$,  choose a linearization $D_\lin\cong D$, thereby identifying 
\[ \wh{A}\cong A\oplus (B^*\otimes C),\ \ \wh{B}\cong B\oplus (A^*\otimes C).\]
Then the fiberwise group actions of the three summands read as: 
\begin{align*}
 (a_1,\omega)\cdot (a,b,c)&=(a+a_1,b,c+\omega(b)),\\ 
  (b_1,\nu)\cdot (a,b,c) &=(a,b+b_1,c+\nu(a)),  \\
  c_1\cdot (a,b,c)&=(a,b,c+c_1).
\end{align*}
Here $a,a_1\in A,\ b,b_1\in B,\ c,c_1\in C$ and $\omega\in B^*\otimes C=\Hom(B,C),\ \nu\in A^*\otimes C=\Hom(A,C)$, all over the same base point in $M$. 

\begin{remark}
Note that the $(\R^2,\cdot)$-action on $G$ makes the Lie group bundle $G\to M$ into a double vector bundle, with 
side bundles $\wh{A},\wh{B}$ and core $C$. The original double vector bundle $D$ is a quotient 
$D=G/G^+$. 
\end{remark}


\subsection{Quotients}
%
By the general theory (see \eqref{eq:condition1}), quotients of a double vector bundle $D\to M$ are classified by graded Lie subalgebra bundles $\h\subset \g$ satisfying
$\P(D)^+\h\cap \g=\h\cap \P(D)^+\g$. This boils down to the following: 
\begin{theorem}\label{th:quot2}
Quotients of a double vector bundle $D\to M$ are classified by bigraded Lie subalgebra bundles
\[ \h=\wh{A}'\oplus \wh{B}'\oplus C'\]
of $ \g=\wh{A}\oplus \wh{B}\oplus C$, 
such that $\iz_{\wh{A}}(B^*\otimes C')=\wh{A}'\cap \iz_{\wh{A}}(B^*\otimes C)$, and similarly 
$\iz_{\wh{B}}(A^*\otimes C')=\wh{B}'\cap \iz_{\wh{A}}(A^*\otimes C)$.
In this case, 
\[ D'=H\cdot M,\ \ D''=D/H\]
where $H=\exp(\h)$ is the subgroup of $G$ exponentiating $\h$. 
\end{theorem}
\begin{proof}
The condition $\P(D)^+\h\cap \g=\h\cap \P(D)^+\g$ is non-trivial only in bidegrees $(-1,0)$ and $(0,-1)$. 
In bidegree $(-1,0)$, it reduces to 
\[ \P(D)^{0,1}  \h^{-1,-1}=\h^{-1,0}\cap \P(D)^{0,1}\g^{-1,-1}.\]
But $\P(D)^{0,1}=B^*,\ \h^{-1,-1}=C',\ \g^{-1,-1}=C,\ \h^{-1,0}=\wh{A}'$, so the condition says 
\[ \iz_{\wh{A}}(B^*\otimes C')=\wh{A}'\cap \iz_{\wh{A}}(B^*\otimes C),\]
which is just the same as $  \iz_{\wh{A}}(B^*\otimes C')\subset \wh{A}'$. Similarly, for the condition in bidegree 
$(0,-1)$. 
\end{proof}
\begin{remark}
Note that a bigraded subbundle $\h\subset \g$ satisfying these conditions 
is automatically closed under brackets. 
\end{remark}
The condition can be interpreted in several other ways. 
For example, given any vector subbundle $C'\subseteq C$ 
we may consider the Lie subalgebra bundle $\g_{C'}\subset \g$ whose sections are the vector fields tangent to $C'$. This subbundle is
\[ \g_{C'}=\iz_{\wh{A}}(B^*\otimes C')\oplus 
\iz_{\wh{B}}(A^*\otimes C')\oplus C'.\]
Comparing with $\h=\wh{A}'\oplus \wh{B}'\oplus C'$, we see: 
\begin{proposition}
A bigraded subbundle  $\h\subset \g$ satisfies the conditions from  Theorem \ref{th:quot2}	
 if and only if 
\[ \h\cap \g_C=\  \g_{C'}\] 
for some $C'\subset C$. 
\end{proposition}

\subsection{Quotient in stages}
According to Theorem \ref{th:stages},  the quotient $D''$ of $D$ may be obtained as an iterated quotient by vector bundle actions. Write  $\h=\wh{A}'\oplus \wh{B}'\oplus C',\ H=\exp\h$ as before. 

Taking a quotient by the action of 
$\exp \h^{-1,-1}\cong \h^{-1,-1}=C'$, one obtains a  double vector bundle  $D/C'$ with side bundles $A,B$ and core $C''=C/C'$. The action of $H$ descends to an action of $H/\exp\h^{-1,-1} \cong \wh{A}'\oplus \wh{B}'$ (a vector bundle seen as an abelian group bundle). Since $\iz_{\wh{A}}(B^*\otimes C')\subset \wh{A}'$ and 
$\iz_{\wh{B}}(A^*\otimes C')\subset \wh{B}'$ act trivially on the quotient, this is really an action of 
$A'\oplus B'$. 
In conclusion, the quotient $D''=D/H$ may be written as 
\[ D''=D/C'/(A'\oplus B').\]
We stress that the action of $A'\oplus B'$ on the quotient $D/C'$  is not canonically determined by $D'$ alone -- this action is encoded by  the Lie algebra $\h\subset \g$.

\subsection{Quotients by wide subbundles}
One simple instance of quotients of double vector bundles  is 
the following. Suppose $D_1\subset D$ is a double subvector bundle with side bundles $A_1,B_1$ and core $C_1$, 
where $A_1=A$ (so that $D_1$ is a `vertically wide' subbundle).
\[
\xymatrix{ {D_1} \ar[r] \ar[d] & B_1\ar[d]\\
	A \ar[r] & M}\]
Then we may simply take the  quotient of $D$ by $D_1$ as vector bundles over $A$. The  result is a double vector bundle 
\[D/\!_v\,D_1\] 
where the subscript signifies the `vertical quotient'. The vertical quotient has side bundles $A,\ B/B_1$ and core $C/C_1$. 

\begin{remark}
The kernel of the quotient map $D\to D''=D/\!_v\,D_1$ is $D'=B_1\times_M C_1$, seen as a double vector bundle whose horizontal side bundle is the zero bundle.  Again, knowledge of the kernel does not suffice to construct the quotient.  The subalgebra $\h$ defining the quotient map is given by 
\[ \h^{-1,-1}=C_1,\ \  \h^{-1,0}=\iz_{\wh{A}}(B^*\otimes C_1),\ \ \h^{0,-1}=\wh{B_1}\] 
where $\wh{B_1}$ denotes the fat bundle for the double vector bundle $D_1$.   
\end{remark}

In a similar way, one may define horizontal quotients $D/\!_h D_2$ by horizontally wide double subvector bundles $D_2\subset D$. 

Suppose now that $D_1\subset D$ is vertically wide, and $D_2\subset D$ is horizontally wide. For $i=1,2$, we denote the side bundles by $A_i,B_i$ and the cores by $C_i$; thus $A_1=A,\ B_2=B$.
\[
\xymatrix{ {D_1} \ar[r] \ar[d] & B_1\ar[d]\\
	A \ar[r] & M}\ \ \ \ \ \ \xymatrix{ {D_2} \ar[r] \ar[d] & B \ar[d]\\
	A_2\ar[r] & M}
\]
Suppose furthermore that $D_1\cap D_2$ is a submanifold of $D$, and hence is a double subvector bundle. Then the image of $D_2$  under the vertical quotient  is a double subvector bundle 
\begin{equation}\label{eq:iterate} D_2/\!_v\,(D_1\cap D_2)\subset D/\!_v\,D_1,\end{equation}
with side bundles $A_2,\ B/B_1$ and core $C_2/C_1\cap C_2$. Hence, the left hand side of 
 \eqref{eq:iterate} is a horizontally wide subbundle, and we may take a horizontal quotient, viewing both sides as vector bundles over $B/B_1$. The result is a double vector bundle 
 \begin{equation}\label{eq:dpprime}
 D''=( D/\!_v\,D_1)/\!_h\,\big(D_2/\!_v\,(D_1\cap D_2)\big)
 \end{equation}
 with side bundles and core given by
\[ A''=A/A_2,\ \ B''=B/B_1,\ \ C''=(C/C_1)/C_1\cap C_2=C/(C_1+C_2).\] 
One may also perform the quotient procedure in the opposite order, starting with a horizontal quotient by $D_2$, followed by a vertical quotient by $D_1/\!_h (D_1\cap D_2)$. Part (b) of the following result
 shows that these two iterated quotients are the same. Part (a) 
 was observed by Jeff Pike. 
\begin{theorem}\label{prop:pike}
Let $D\to M$ be a double vector bundle, with side bundles $A,B$ and core $C$. 
\begin{enumerate}
\item Every quotient $D''$ of $D$ is of the form \eqref{eq:dpprime}, 
where $D_1,D_2\subset D$ is a pair of double vector bundles, with $D_1$ wide over $A$ and $D_2$ wide over $B$, and $D_1\cap D_2$ a submanifold. 
\item 
There is a canonical isomorphism of double vector bundles
\[ ( D/\!_v\,D_1)/\!_h\,\big(D_2/\!_v\,(D_1\cap D_2)\big)  \cong 
( D/\!_h\,D_2)/\!_v\,\big(D_1/\!_h\,(D_1\cap D_2)\big) .\]
\item 
For a given quotient as in (a), the 
pair $D_1,D_2$ is uniquely determined if we require that $C_1=C_2$. 
\end{enumerate}
\end{theorem}
\begin{proof}
Given a quotient $\pi\colon D\to D''$ with 	kernel $D'$, let $\h=\wh{A}'\oplus \wh{B}'\oplus C' \subset \g$ be as in Theorem \ref{th:quot2}. 
Let 
\[ D_1=\wh{B}'\cdot A,\ \ D_2=\wh{A}'\cdot B\]
be the flow-outs under the action of the corresponding subgroups of $H$. 
Choosing   a linearization $D\cong A\times_M B \times_MC $ and the corresponding description of the action 
of $\wh{A}$ on $D$, we see that 
\[ D_1\cong A\times_M B'\times_M C',\ \ \ D_2\cong A'\times_M B\times_M C'.\] 
This shows that $D_1$ is a double subvector bundle with  $A_1=A,\,B_1=B',\,C_1=C'$, and $D_2$ is a double subvector bundle with $A_2=A',\,B_2=B,\,C_2=C'$. 
Write $\h=\h_1+\h_2$, where  
\[ \h_1=\iz_{\wh{A}}(B^*\otimes C')\oplus \wh{B}'\oplus C',\ \ 
\h_2=\wh{A}'\oplus \iz_{\wh{B}}(A^*\otimes C') \oplus C',\]
are normal subalgebra bundle of $\h$. Let $H_1=\exp \h_1,\ \ H_2=\exp\h_2$ be the corresponding group bundles. 
Taking a vertical quotient by $D_1$ amounts to taking the quotient by the action of $H_1=\exp\h_1 \subset H$; 
the subsequent horizontal quotient by $D_2/\!_v(D_1\cap D_2)$ coincides with the quotient by the action of $H_2/H_1\cap H_2\subset H/H_1$. Clearly, this iterated quotient is the same as the quotient by $H$. This shows 
\[ D''=( D/\!_v\,D_1)/\!_h\,\big(D_2/\!_v\,(D_1\cap D_2)\big),\]
completing the proof of (a), with the canonical choice of $D_1,D_2$ suggested in (c). Part (b) also follows for this particular choice of $D_1,D_2$.  Given a more general choice of $D_1,D_2$, with 
$D_1\subset D$ vertically wide and $D_2\subset D$ horizontally wide,   put $C'=C_1+C_2$. Replacing $D_1,D_2$ with 
their flow-outs under the core action $\ti{D}_1=C'\cdot D_1,\ \ti{D}_2=C'\cdot D_2$ does not change the side bundles, but replaces the both cores with $C'$. The iterated quotients by the actions of these bundles are not changed either. 
This proves (b), as well as (c). 
\end{proof}

\section{Application: double normal bundles}\label{sec:dn}\label{sec:dnb}

\subsection{Double tangent bundle}
An important example of a double vector bundle is the \emph{double tangent bundle} 
\[
\xymatrix{ {TTM} \ar[r] \ar[d] & TM \ar[d]\\
	TM \ar[r] & M}
\]
As stated in Example \ref{ex:rthtangent},  this
may be defined as
\[ TTM=\Hom_\alg(C^\infty_\pol(M),\AA_1\otimes \AA_1),\]
where $\AA_{1}$ is the commutative unital algebra with generator $\epsilon$ and relation $\epsilon^2=0$. 
Both side bundles $A,B$ are equal to $TM$, and the core $C$ is yet another copy of $TM$. In terms of the identification  $TM=\Hom_\alg(C^\infty(M),\AA_1)$, the inclusion maps for these three subbundles correspond to the algebra morphisms $\AA_1\to \AA_1\otimes \AA_1$ taking the 
generator $\epsilon$ to $\epsilon\otimes 1$, $1\otimes \epsilon$, and $\epsilon\otimes \epsilon$, respectively.

The bigraded Lie algebra bundle $\V(TTM)^{<0}$ has components in bidegrees $(-1,0),(0,-1),(-1,-1)$. It  
admits a very concrete description in terms of lifts of vector fields. Recall that every  vector field $X\in \mf{X}(M)$ has two kinds of lifts to the tangent bundle, the \emph{vertical lift} and the \emph{tangent lift} 
\[ X_v\in\mf{X}_\pol(TM)^{-1},\ \ 
X_t\in \mf{X}_\pol(TM)^{0}.
\] 
The vertical lift $X_v$ is defined by the canonical isomorphism $\mf{X}_\pol(V)^{-1}\cong \Gamma(V)$ for vector bundles $V\to M$ (here $V=TM$),  extending sections of $V\subset TV|_M=V\oplus TM$ to fiberwise constant vertical vector fields. 
The tangent lift $X_t$ is the vector field 
whose local flow is obtained by applying the tangent lift to the local flow of $X$.
This vector field is tangent to the zero section, with 
$X_t|_M=X$. There are also two lift operations for functions $f\in C^\infty(M)$,
\[  f_v\in C^\infty_\pol(TM)^{0},\ \ f_t\in C^\infty_\pol(TM)^1.\]
The vertical lift $f_v$ is simply the  pullback, and the tangent lift $f_t$ is the exterior differential regarded as a function on $TM$. These lifts are related by various identities, most easily remembered by examining the degree of homogeneity; for example 
\[ (fX)_t=f_t X_v+f_vX_t,\ [X,Y]_v=[X_t,Y_v]=[X_v,Y_t],\ (Xf)_t=X_t f_t.\]

For a double tangent bundle $TTM$, we may combine these two operations and obtain four kinds of lifts of functions \[  f_{tt}\in C^\infty_\pol(TTM)^{0,0},\ \  f_{tv}\in C^\infty_\pol(TTM)^{0,1},\ \  f_{vt}\in C^\infty_\pol(TTM)^{1,0},\ \  f_{vv}\in C^\infty_\pol(TTM)^{1,1},\] and also 
four kinds of lifts of vector fields
\[  X_{tt}\in\mf{X}_\pol(TTM)^{0,0},\  \  X_{tv}\in\mf{X}_\pol(TTM)^{0,-1},\  \  X_{vt}\in\mf{X}_\pol(TTM)^{-1,0},\ \  X_{vv}\in\mf{X}_\pol(TTM)^{-1,-1}.\]
Let us  write $A,B\cong TM$ for the horizontal and vertical side bundles of $TTM$, and $C\cong TM$ for the core. The lifts $X_{vv}$ realize the isomorphism 
\[ \mf{X}_\pol(TTM)^{-1,-1}\cong \Gamma(C)\]
(`core sections'). Furthermore, we obtain splittings of the  exact sequences \eqref{eq:ahat} and \eqref{eq:bhat}  for the fat bundles, on the level of sections: In fact, 
\[ \mf{X}_\pol(TTM)^{0,-1}\cong \Gamma(B\oplus (A^*\otimes C)),\]
where the first summand $\Gamma(B)$ corresponds to  vector fields of the form $X_{tv}$, and the second summand 
$\Gamma(A^*\otimes C)$ is realized as linear combinations of products $f_{tv}X_{vv}$. 
Similarly, 
\[ \mf{X}_\pol(TTM)^{-1,0}\cong \Gamma(A\oplus (B^*\otimes C)).\]

\subsection{Double normal bundles}
Recall that two embedded submanifolds $N_1,N_2$ of a manifold $M$ \emph{intersect cleanly} if the intersection ${N}=N_1\cap N_2$ 
is a submanifold, with $T{N}=TN_1\cap TN_2$. In this situation, there is a double vector bundle
called the \emph{double normal bundle} 
\cite{me:wei}, 
\begin{equation}\label{eq:dnb}
\xymatrix{ {\nu(M,N_1,N_2)} \ar[r] \ar[d] & \nu(N_2,{N}) \ar[d]\\
	\nu(N_1,{N}) \ar[r] & {N}}
\end{equation}
The side bundles of the double normal bundle are the usual normal bundles of $N$ inside $N_1.N_2$, and the core is
given by   $TM|_N/(TN_1|_N+TN_2|_N)$. Note that the rank of the core is the \emph{excess} of the clean intersection, 
i.e., the core is the zero bundle if and only if the intersection is transverse. Interchanging the roles of $N_1,N_2$ gives a 
symmetry property 
\begin{equation}\label{eq:symmetry} \nu(M,N_1,N_2)\on \cong \on{flip}(\nu(M,N_2,N_1)),\end{equation} 
where `$\on{flip}$' interchanges the horizontal and vertical vector bundle structures. 

The double normal bundle may be described 
most simply as  an iterated normal bundle
$\nu(\nu(M,N_1),\nu(N_2,{N}))$, but this approach has some drawbacks -- for example, 
the property  \eqref{eq:symmetry} is not evident. A second approach, 
from which the symmetry \eqref{eq:symmetry}  is obvious, uses the bifiltration of 
$C^\infty(M)$  by the ideals $C^\infty(M)_{(i_1,i_2)}$ of functions vanishing to order $i_1$ on $N_1$ and to order $i_2$ on $N_2$. Letting $\on{gr}(C^\infty(M))$ denotes the associated bi-graded algebra, one has that 
\[ \nu(M,N_1,N_2)=\Hom_\alg(\on{gr}(C^\infty(M)),\R).\]
We shall now give a third description as a subquotient of the double tangent bundle $TTM$, analogous to the description $\nu(M,N)=TM|_N/TN$ 
of the usual normal bundle. Let 
\[ \varphi\colon TTM\to TM\times_M TM\]
be the map given by the horizontal and vertical base projection, and put
\[ D=\varphi^{-1} (TN_1|_{N}\times_{N} TN_2|_{N})\subset TTM.\] 
This is double subvector bundle 
\[
\xymatrix{ D \ar[r] \ar[d] & TN_2|_{N}\ar[d]\\
	TN_1|_{N}\ar[r] & {N}}
\]
with core $TM|_{N}$. The sections of the 
Lie algebra bundle $\g=\V(D)^{<0}$ are the restrictions of polynomial vector fields that are tangent to $D$:
\begin{align*}
\Gamma(\g^{-1,-1})&=\{X_{vv}|_D\colon X\mbox{ arbitrary}\},\\
\Gamma(\g^{0,-1})&=\{X_{tv}|_D\colon X|_{N}\in \Gamma(TN_1|_{N})\}+C^\infty_\pol(D)^{1,0}\cdot \Gamma(\g^{-1,-1}),\\
\Gamma(\g^{-1,0})&=\{X_{vt}|_D\colon X|_{N}\in \Gamma(TN_2|_{N})\}
+C^\infty_\pol(D)^{0,1}\cdot \Gamma(\g^{-1,-1}).
\end{align*} 
(Here $C^\infty_\pol(D)^{0,1},\ C^\infty_\pol(D)^{1,0}$ 
are realized as restrictions of functions of the form $f_{vt}$, respectively $f_{tv}$.) We hence see that, similarly to $TTM$, the exact sequences for the fat bundles of the double subvector bundle $D$ split on the level of sections. 
The subbundle $\h\subset \g$ defining the quotient map is described on the level of sections as
\begin{align*}
\Gamma(\h^{-1,-1})&=\{X_{vv}|_D\colon X|_{N}\in \Gamma(TN_1|_{N}+TN_2|_{N})\},\\
\Gamma(\h^{0,-1})&=\{X_{vt}|_D\colon X|_{N}\in \Gamma(T{N})\}+C^\infty_\pol(D)^{1,0}\cdot \Gamma(\h^{-1,-1}),\\
\Gamma(\h^{-1,0})&=\{X_{tv}|_D\colon X|_{N}\in \Gamma(T{N})\}
+C^\infty_\pol(D)^{0,1}\cdot \Gamma(\h^{-1,-1}).
\end{align*} 
Thus, with the group bundle $H=\exp\h$, we have $\nu(M,N_1,N_2)=D/H$. 

Alternatively, we may use the method from Theorem \ref{prop:pike} to define the quotient. For $i=1,2$, let 
$\varphi_i\colon TTN_i\to TN_i\times_{N_i} TN_i$  be projection to the two side bundles. The pre-images 
\[ D_1=\varphi_1^{-1}(TN_1|_{N}\times_{N} T{N})\subset TTN_1,\ \ \ D_2=\varphi_2^{-1}(T{N}\times_{N} TN_2|_{N})\subset TTN_2\] are double subvector bundles 
\[
\xymatrix{ D_1 \ar[r] \ar[d] & T{N}\ar[d]\\
	TN_1|_{N}\ar[r] & {N}}\ \ \ \ \ \ 
\xymatrix{D_2 \ar[r] \ar[d] & TN_2|_{N}\ar[d]\\
	T{N}\ar[r] & {N}}\]
with cores $C_1=TN_1|_{N},\ C_2=TN_2|_{N}\subset C$, respectively. We may describe $\nu(M,N_1,N_2)$ as the vertical quotient by 
$D_1$ followed by the horizontal quotient by $D_2/\!_v\,(D_1\cap D_2)$; according to Proposition \ref{prop:pike} this is canonically isomorphic to the horizontal quotient by $D_2$ followed by the vertical quotient by $ D_1/\!_h\,(D_1\cap D_2)$.

\section{Example: Weighted normal bundles for order $2$ weightings} \label{sec:r2}
Let $M$ be a manifold. In \cite{loi:wei}, we developed a theory of 
\emph{weightings along submanifolds $N\subset M$}, a notion introduced in lecture notes of Melrose \cite{mel:cor}  under the name of quasi-homogeneous structure. A weighting is described in terms of a suitable filtration 
\[ C^\infty(M)=C^\infty(M)_{(0)}\supseteq C^\infty(M)_{(1)}\supseteq\cdots\] 
of the algebra of smooth functions on $M$. Roughly, the filtration amounts to an assignment of weights to coordinate functions, making sense of a weighted order of vanishing along $N$.  An upper bound for the weights is called an \emph{order of the weighting}. Weightings determine a filtration 
\[ 0_N=F_0\subset F_{-1}\subset \cdots \subset F_{-r}=\nu(M,N)\] 
of 
the normal bundle $\nu(M,N)$, and a main result in \cite{loi:wei} is the construction of a  \emph{weighted normal bundle} 
 \[\nu_\W(M,N)\to N,\] which is a graded bundle over $N$ with linear approximation \[\nu_\W(M,N)_\lin=\bigoplus_{i=1}^r 
 F_{-i}/F_{-i+1}.\] 
The weighted normal bundle admits an algebraic description as $\Hom_\alg(\on{gr}(C^\infty(M)),\R)$, where 
$\on{gr}(C^\infty(M))$ is the associated graded algebra for the given filtration. But as shown in \cite{loi:wei}, it may also be described as a subquotient of the $r$-th tangent bundle $T_rM$. 

Let us describe the quotient procedure for the special case of  weightings of order $r=2$. Here the notion of 
weighting is simpler \cite{me:eul}, since it is  fully determined by the subbundle $F_{-1}=F$.  Letting $\ti{F}\subset TM|_N$ be the pre-image of $F$ under the quotient map $TM|_N\to \nu(M,N)$, the filtration on functions is generated by letting $C^\infty(M)_{(1)}$ be all functions vanishing along $N$, and $C^\infty(M)_{(2)}\subset C^\infty(M)_{(1)}$ those functions whose differential vanishes on $\ti{F}$.

Each $f\in C^\infty(M)$ has three kinds of lifts 
\[ f^{(i)}\in C^\infty(T_2M)^i,\ i=0,1,2.\] In terms of the 
definition $T_2M=\on{Hom}_\alg(C^\infty(M),\AA_2)$, these are defined by 
\[ \mathsf{u}(f)=f^{(0)}(\mathsf{u})+f^{(1)}(\mathsf{u})\epsilon+f^{(2)}(\mathsf{u}){\epsilon^2},\]
for $\mathsf{u}\in T_2M$. Alternatively, thinking of elements of $T_2M$ as 2-jets of curves $\gamma\colon \R\to M$, 
\[ f^{(0)}(j^2(\gamma))=f(\gamma(0)),\ \ f^{(1)}(j^2(\gamma))=\f{d}{d t}|_{t=0}f(\gamma(t)),\ \ f^{(2)}(j^2(\gamma))=2\f{d^2}{d t^2}|_{t=0}f(\gamma(t)).\]
For vector fields $X\in\mf{X}(M)$ we have lifts \cite{mor:lif}
\[ X^{(-i)}\in \mf{X}_\pol(T_2M)^{-i},\ i=0,1,2,\] given in terms of the action on functions by 
\[ X^{(-i)}f^{(j)}=(Xf)^{(j-i)}\]
where the right hand side is taken to be zero if $j-i<0$. The lifts $X^{(-2)}$ define a vector bundle action of $TM$ on $T_2M$. 
Now let 
\[ Q=T_2M\big|_{\wt{F}}\subset T_2M\] 
be the pre-image of $\ti{F}$ under the map $T_2M\to TM$.  (Note that this is analogous to the definition of $D\subset TTM$ in the previous section.) Then $Q\to N$ is a graded subbundle of $T_2M\to M$, with linear approximation $Q_\lin=\ti{F}\oplus TM|_N$. 
%

Let us now consider the Lie algebra bundle $\g=\V(Q)^{<0}$.  
The pairing of homogeneous polynomials of degree $1$ on $E$ with homogeneous vector fields of degree $-1$, followed by restriction to $M\subset E$, defines a non-degenerate pairing between $\P(E)^1$ and $E_\lin^{-1}$; hence 
	$\P(E)^1\cong  (E_\lin^{-1})^*$. 
	From the exact sequence 
	\eqref{eq:exactve}, we see that  
	\[ \g^{-2}=E_\lin^{-2},\] 
	while 
	$\g^{-1}$ fits into an exact sequence 
	\begin{equation}\label{eq:exactr2} 0_M\to (E_\lin^{-1})^*\otimes \g^{-2}\to \g^{-1}\to E_\lin^{-1}\to 0_M.\end{equation}
In terms of its space of sections, 
\begin{align*}
\Gamma(\g^{-2})&=\{X^{(-2)}|_Q\colon X\mbox{ arbitrary}\}\\
\Gamma(\g^{-1})&=\{\Big(X^{(-1)}+\sum_\nu f_\nu^{(1)} Y_\nu^{(-2)}\Big)|_Q\colon X|_N\in \Gamma(\ti{F})\ 
f_\nu,Y_\nu \mbox{ arbitrary}\}.
\end{align*}
In particular, the exact sequence \eqref{eq:exactr2} splits on the level of sections. 
Let $\h\subset \g$ be the Lie subalgebra bundle given on the level of sections by 
\begin{align*}
\Gamma(\h^{-2})&=\{X^{(-2)}|_Q\colon X|_N\in \Gamma(\ti{F})\}\\
\Gamma(\h^{-1})&=\{\Big(X^{(-1)}+\sum_\nu f_\nu^{(1)} Y_\nu^{(-2)}\Big)|_Q\colon X|_N\in \Gamma(TN)\ 
f_\nu\mbox{ arbitrary},\ Y_\nu|_N\in  \Gamma(\ti{F})\}.
\end{align*}
The weighted normal bundle is the quotient $\nu_\W(M,N)=Q/H$, with kernel
$H\cdot M=\wt{F}\cdot T_2N\subset T_2M$ (using the vector bundle action of $\wt{F}\subset TM|_N$ on 
$T_2M|_N$). We hence have the exact sequence 
\[ 0_N\to \wt{F}\cdot T_2N\to T_2M\big|_{\wt{F}}\to \nu_\W(M,N)\to 0_N\]
for the weighted normal bundle, generalizing the defining sequence 
\[ 0_N\to TN\to TM|_N\to \nu(M,N)\to 0_N\] for the usual normal bundle. Observe also that 
\[ \nu_\W(M,N)_\lin=(T_2M\big|_{\wt{F}})_\lin/(\wt{F}\cdot T_2N)_\lin=\ti{F}/TN\oplus TM|_N/\ti{F}=F\oplus \nu(M,N)/F,\]
as expected.

\bibliographystyle{amsplain} 
\def\cprime{$'$} \def\polhk#1{\setbox0=\hbox{#1}{\ooalign{\hidewidth
			\lower1.5ex\hbox{`}\hidewidth\crcr\unhbox0}}} \def\cprime{$'$}
\def\cprime{$'$} \def\cprime{$'$} \def\cprime{$'$} \def\cprime{$'$}
\def\polhk#1{\setbox0=\hbox{#1}{\ooalign{\hidewidth
			\lower1.5ex\hbox{`}\hidewidth\crcr\unhbox0}}} \def\cprime{$'$}
\def\cprime{$'$} \def\cprime{$'$} \def\cprime{$'$} \def\cprime{$'$}
\providecommand{\bysame}{\leavevmode\hbox to3em{\hrulefill}\thinspace}
\providecommand{\MR}{\relax\ifhmode\unskip\space\fi MR }
\providecommand{\MRhref}[2]{%
	\href{http://www.ams.org/mathscinet-getitem?mr=#1}{#2}
}
\providecommand{\href}[2]{#2}

\end{document}